\documentclass[12pt,reqno]{amsart}
\usepackage{amscd,amssymb,amsmath,multicol,float,scalefnt,cancel,array,stmaryrd,graphicx,tikz}
\restylefloat{table}
\newtheorem{thm}[equation]{Theorem}
\numberwithin{equation}{section}

\newtheorem{defin}[equation]{Definition}

\newtheorem{prop}[equation]{Proposition}

\newtheorem{fig}[equation]{Figure}

\begin{document}
\raggedbottom \voffset=-.7truein \hoffset=0truein \vsize=8truein
\hsize=6truein \textheight=8truein \textwidth=6truein
\baselineskip=18truept

\def\mapright#1{\ \smash{\mathop{\longrightarrow}\limits^{#1}}\ }
\def\mapleft#1{\smash{\mathop{\longleftarrow}\limits^{#1}}}
\def\mapup#1{\Big\uparrow\rlap{$\vcenter {\hbox {$#1$}}$}}
\def\mapdown#1{\Big\downarrow\rlap{$\vcenter {\hbox {$\ssize{#1}$}}$}}
\def\mapne#1{\nearrow\rlap{$\vcenter {\hbox {$#1$}}$}}
\def\mapse#1{\searrow\rlap{$\vcenter {\hbox {$\ssize{#1}$}}$}}
\def\mapr#1{\smash{\mathop{\rightarrow}\limits^{#1}}}
\def\ss{\smallskip}
\def\s{\sigma}
\def\l{\lambda}
\def\vp{v_1^{-1}\pi}
\def\at{{\widetilde\alpha}}
\def\sm{\wedge}
\def\la{\langle}
\def\ra{\rangle}
\def\ev{\text{ev}}
\def\od{\text{od}}
\def\on{\operatorname}
\def\ol#1{\overline{#1}{}}
\def\spin{\on{Spin}}
\def\cat{\on{cat}}
\def\lbar{\ell}
\def\qed{\quad\rule{8pt}{8pt}\bigskip}
\def\ssize{\scriptstyle}
\def\a{\alpha}
\def\bz{{\Bbb Z}}
\def\Rhat{\hat{R}}
\def\im{\on{im}}
\def\ct{\widetilde{C}}
\def\ext{\on{Ext}}
\def\sq{\on{Sq}}
\def\eps{\epsilon}
\def\ar#1{\stackrel {#1}{\rightarrow}}
\def\br{{\bold R}}
\def\bC{{\bold C}}
\def\bA{{\bold A}}
\def\bB{{\bold B}}
\def\bD{{\bold D}}
\def\bC{{\bold C}}
\def\bh{{\bold H}}
\def\bQ{{\bold Q}}
\def\bP{{\bold P}}
\def\bx{{\bold x}}
\def\bo{{\bold{bo}}}
\def\dh{\widehat{d}}
\def\si{\sigma}
\def\Vbar{{\overline V}}
\def\dbar{{\overline d}}
\def\wbar{{\overline w}}
\def\Sum{\sum}
\def\tfrac{\textstyle\frac}
\def\tb{\textstyle\binom}
\def\Si{\Sigma}
\def\w{\wedge}
\def\equ{\begin{equation}}
\def\b{\beta}
\def\G{\Gamma}
\def\L{\Lambda}
\def\g{\gamma}
\def\d{\delta}
\def\k{\kappa}
\def\psit{\widetilde{\Psi}}
\def\tht{\widetilde{\Theta}}
\def\psiu{{\underline{\Psi}}}
\def\thu{{\underline{\Theta}}}
\def\aee{A_{\text{ee}}}
\def\aeo{A_{\text{eo}}}
\def\aoo{A_{\text{oo}}}
\def\aoe{A_{\text{oe}}}
\def\vbar{{\overline v}}
\def\endeq{\end{equation}}
\def\sn{S^{2n+1}}
\def\zp{\bold Z_p}
\def\cR{{\mathcal R}}
\def\P{{\mathcal P}}
\def\cQ{{\mathcal Q}}
\def\cj{{\cal J}}
\def\zt{{\bold Z}_2}
\def\bs{{\bold s}}
\def\bof{{\bold f}}
\def\bq{{\bold Q}}
\def\be{{\bold e}}
\def\Hom{\on{Hom}}
\def\ker{\on{ker}}
\def\kot{\widetilde{KO}}
\def\coker{\on{coker}}
\def\da{\downarrow}
\def\colim{\operatornamewithlimits{colim}}
\def\zphat{\bz_2^\wedge}
\def\io{\iota}
\def\om{\omega}
\def\Prod{\prod}
\def\e{{\cal E}}
\def\zlt{\Z_{(2)}}
\def\exp{\on{exp}}
\def\abar{{\overline a}}
\def\xbar{{\overline x}}
\def\ybar{{\overline y}}
\def\zbar{{\overline z}}
\def\mbar{{\overline m}}
\def\nbar{{\overline n}}
\def\sbar{{\overline s}}
\def\kbar{{\overline k}}
\def\bbar{{\overline b}}
\def\et{{\widetilde E}}
\def\ni{\noindent}
\def\tsum{\textstyle \sum}
\def\coef{\on{coef}}
\def\den{\on{den}}
\def\lcm{\on{l.c.m.}}
\def\vi{v_1^{-1}}
\def\ot{\otimes}
\def\psibar{{\overline\psi}}
\def\thbar{{\overline\theta}}
\def\mhat{{\hat m}}
\def\exc{\on{exc}}
\def\ms{\medskip}
\def\ehat{{\hat e}}
\def\etao{{\eta_{\text{od}}}}
\def\etae{{\eta_{\text{ev}}}}
\def\dirlim{\operatornamewithlimits{dirlim}}
\def\gt{\widetilde{L}}
\def\lt{\widetilde{\lambda}}
\def\st{\widetilde{s}}
\def\ft{\widetilde{f}}
\def\sgd{\on{sgd}}
\def\lfl{\lfloor}
\def\rfl{\rfloor}
\def\ord{\on{ord}}
\def\gd{{\on{gd}}}
\def\rk{{{\on{rk}}_2}}
\def\nbar{{\overline{n}}}
\def\MC{\on{MC}}
\def\lg{{\on{lg}}}
\def\cH{\mathcal{H}}
\def\cS{\mathcal{S}}
\def\cP{\mathcal{P}}
\def\N{{\Bbb N}}
\def\Z{{\Bbb Z}}
\def\Q{{\Bbb Q}}
\def\R{{\Bbb R}}
\def\C{{\Bbb C}}

\def\mo{\on{mod}}
\def\xt{\times}
\def\notimm{\not\subseteq}
\def\Remark{\noindent{\it  Remark}}
\def\kut{\widetilde{KU}}

\def\*#1{\mathbf{#1}}
\def\0{$\*0$}
\def\1{$\*1$}
\def\22{$(\*2,\*2)$}
\def\33{$(\*3,\*3)$}
\def\ss{\smallskip}
\def\ssum{\sum\limits}
\def\dsum{\displaystyle\sum}
\def\la{\langle}
\def\ra{\rangle}
\def\on{\operatorname}
\def\proj{\on{proj}}
\def\od{\text{od}}
\def\ev{\text{ev}}
\def\o{\on{o}}
\def\U{\on{U}}
\def\lg{\on{lg}}
\def\a{\alpha}
\def\bz{{\Bbb Z}}
\def\eps{\varepsilon}
\def\bc{{\bold C}}
\def\bN{{\bold N}}
\def\bB{{\bold B}}
\def\bW{{\bold W}}
\def\nut{\widetilde{\nu}}
\def\tfrac{\textstyle\frac}
\def\b{\beta}
\def\G{\Gamma}
\def\g{\gamma}
\def\zt{{\Bbb Z}_2}
\def\zth{{\bold Z}_2^\wedge}
\def\bs{{\bold s}}
\def\bx{{\bold x}}
\def\bof{{\bold f}}
\def\bq{{\bold Q}}
\def\be{{\bold e}}
\def\lline{\rule{.6in}{.6pt}}
\def\xb{{\overline x}}
\def\xbar{{\overline x}}
\def\ybar{{\overline y}}
\def\zbar{{\overline z}}
\def\ebar{{\overline \be}}
\def\nbar{{\overline n}}
\def\ubar{{\overline u}}
\def\bbar{{\overline b}}
\def\et{{\widetilde e}}
\def\lf{\lfloor}
\def\rf{\rfloor}
\def\ni{\noindent}
\def\ms{\medskip}
\def\Dhat{{\widehat D}}
\def\what{{\widehat w}}
\def\Yhat{{\widehat Y}}
\def\abar{{\overline{a}}}
\def\minp{\min\nolimits'}
\def\sb{{$\ssize\bullet$}}
\def\mul{\on{mul}}
\def\N{{\Bbb N}}
\def\Z{{\Bbb Z}}
\def\S{\Sigma}
\def\Q{{\Bbb Q}}
\def\R{{\Bbb R}}
\def\C{{\Bbb C}}
\def\Xb{\overline{X}}
\def\eb{\overline{e}}
\def\notint{\cancel\cap}
\def\cS{\mathcal S}
\def\cR{\mathcal R}
\def\el{\ell}
\def\TC{\on{TC}}
\def\GC{\on{GC}}
\def\wgt{\on{wgt}}
\def\wpt{\widetilde{p_2}}
\def\dstyle{\displaystyle}
\def\Om{\Omega}
\def\ds{\dstyle}
\def\tz{tikzpicture}
\def\zcl{\on{zcl}}
\def\bd{\bold{d}}
\def\Vb#1{{\overline{V_{#1}}}}
\title
{Geodesic complexity for non-geodesic spaces}
\author{Donald M. Davis}
\address{Department of Mathematics, Lehigh University\\Bethlehem, PA 18015, USA}
\email{dmd1@lehigh.edu}

\date{May 27, 2021}

\keywords{geodesic, configuration space, topological robotics}
\thanks {2000 {\it Mathematics Subject Classification}: 53C22, 55R80, 55M30, 68T40.}

\maketitle
\begin{abstract} We define the notion of near geodesic between points of a metric space when no geodesic exists, and use this to extend Recio-Mitter's notion of geodesic complexity to non-geodesic spaces. This has potential application to topological robotics. We determine explicit near geodesics and geodesic complexity in a variety of cases.
\end{abstract}
\section{Introduction}\label{intro} In \cite{RM}, Recio-Mitter defined the geodesic complexity $\GC(X)$ of a metric space $X$  to be the smallest number $k$ such that $X\times X$ can be partitioned into $k+1$ locally compact sets $E_i$, $0\le i\le k$,  such that on each $E_i$ there is a continuous map $s_i:E_i\to PX$, called a geodesic motion planning rule (GMPR) on $E_i$, such that, for all $(x_0,x_1)\in E_i$, $s_i(x_0,x_1)$ is a (minimal) geodesic from $x_0$ to $x_1$. This was an analogue of Farber's notion of topological complexity $TC(X)$ (\cite{Far}), which applied to any topological space $X$ and did not require that $s_i(x_0,x_1)$ be a geodesic. Clearly $\TC(X)\le \GC(X)$ for any metric space $X$. These notions are of particular interest if $X$ is a space of configurations of one or more robots.

A {\it geodesic} space is one in which for all pairs $(x_0,x_1)$ of points, there is a geodesic from $x_0$ to $x_1$. According to the definition in \cite{RM}, $\GC(X)=\infty$ if $X$ is not geodesic. In \cite{D} and \cite{DHR}, some non-geodesic spaces $X$ were replaced by homotopically-equivalent geodesic spaces, whose GC was computed and interpreted as representing $\GC(X)$. This seems reasonable since TC is a homotopy invariant.

Let $F(X,2)$ denote the space of ordered pairs of distinct points of $X$, with the induced metric from $X\times X$.  In \cite{D}, we replaced the non-geodesic space $F(\R^n,2)$ by the homotopically-equivalent geodesic space $F_\eps(\R^n,2)$, consisting of points $(x_0,x_1)$ satisfying $d(x_0,x_1)\ge\eps$. We determined explicit geodesics in $F_\eps(\R^n,2)$, but the work was quite complicated.


 For a topological space $Y$, $P(Y)=Y^I$ denotes the free path space of continuous maps $I\to Y$, with the compact-open topology, and $P(Y;y_0,y_1)$ the subspace consisting of paths from $y_0$ to $y_1$.

\begin{defin}\label{nogeo} Let $X$ be a metric space whose completion $\Xb$ is geodesic. For $x_0,x_1\in X$, a {\em near-geodesic} from $x_0$ to $x_1$ is a continuous map $\phi:I\to P(\Xb;x_0,x_1)$ satisfying
\begin{itemize}
\item[i.] $\phi(0)$ is a geodesic in $\Xb$ from $x_0$ to $x_1$;
\item[ii.] $\phi((0,1])\subset P(X;x_0,x_1)$;
\item[iii.] if $s_n\to0$, then $\operatorname{length}(\phi(s_n))\to\operatorname{length}(\phi(0))$.\end{itemize}\end{defin}

Here length of a path is defined (as in \cite{RM}) to be the supremum of sums of distances between successive points on the path. Item (iii) of the above definition guarantees that for values of $s$ close enough to 0 the paths $\phi(s)$ in $X$ are good approximations to the geodesic in $\Xb$. If there is a geodesic in $X$ from $x_0$ to $x_1$, one can use the constant homotopy, but we still call it a near geodesic for uniformity.

\begin{defin}\label{NGMPRdef}
For $E\subset X\times X$, a {\em near-geodesic motion planning rule} (NGMPR) on $E$ is a continuous map $\Phi$ from $E$ to $P(\Xb)^{I}$ such that, for all $(x_0,x_1)\in E$, $\Phi(x_0,x_1)$ is a near-geodesic from $x_0$ to $x_1$, with the additional proviso that $\Phi(x_0,x_1)(0)$ is a geodesic in $X$ if one exists. The {\em geodesic complexity} $\GC(X)$ is defined as the smallest $k$ such that $X\times X$ can be partitioned into locally compact sets $E_0,\ldots,E_k$ such that each $E_i$ has an NGMPR. \end{defin}

The following elementary proposition suggests that this is a good definition.
\begin{prop}\label{equal}If $X$ is geodesic, then the definition of $\GC(X)$ in Definition \ref{NGMPRdef}  agrees with that in \cite{RM}.\end{prop}
\begin{proof} If $s$ is a GMPR on $E$, then using the constant homotopy yields an NGMPR on $E$. If $\Phi$ is an NGMPR on $E$, then $\Phi|E\times 0$ is a GMPR.\end{proof}

One might prefer that the ``additional proviso'' in Definition \ref{NGMPRdef} be omitted. We have included it so that Proposition \ref{equal} is true. It is conceivable that this proposition might still be true if that requirement were omitted, but we have not been able to prove this.
Note that $\TC(X)\le \GC(X)$ since if $\Phi$ is an NGMPR on $E$, then the map $E\to P(X)$ defined by $(x_0,x_1)\mapsto \Phi(x_0,x_1)(\frac12)$ is a motion planning rule on $E$.

In this paper, we show that  $\GC(X)=\TC(X)$ for the following non-geodesic spaces $X$ by constructing explicit NGMPRs.
\begin{itemize}
\item
$\R^n-Q$ with $n\ge2$ and $Q$ a finite subset,
\item $F(\R^n,2)$ with $n\ge2$,
\item  $F(\R^n-\{x_0\},2)$ with $n\ge2$,
\item the unordered configuration space $C(\R^2-\{x_0\},2)$,  \item $F(Y,2)$, where $Y$ is a graph with exactly one essential vertex, of order 3.
\end{itemize}
 In Section \ref{sec4}, we study $\GC(X)$ for $X=F(\R^n-Q,2)$ if $Q$ is a finite subset with at least two points.  This might be a case in which $\GC(X)>\TC(X)$.

We thank David Recio-Mitter  and a referee for  insightful comments.
\section{Some examples with $\GC=\TC$}
In this section, we determine explicit NGMPRs in the first four of the five examples with $\GC(X)=\TC(X)$ listed above. In each of these examples, we find it convenient to let $g:I\to I$ denote a continuous function such as $g(t)=\sin(\pi t)$ or $g(t)=1-|2t-1|$  satisfying $g^{-1}(0)=\{0,1\}$.

\begin{defin} The nogeo set of $X$ is  the set of $(x_0,x_1)\in X\times X$ for which there is no (minimal) geodesic from $x_0$ to $x_1$. Its complement is the geoset of $X$.\end{defin}

\begin{thm} \label{thm1} For $n\ge2$ and $x_0\in\R^n$, $$\GC(\R^n-\{x_0\})=\TC(\R^n-\{x_0\})=\begin{cases}1&n\text{ even}\\ 2&n\text{ odd.}\end{cases}$$\end{thm}
\begin{proof} Here $\Xb=\R^n$.  By a linear homeomorphism of $\R^n$, we may assume $x_0=0$. The nogeo set $E$ is $\{(a,b):a=\l b,\ \l<0\}$. We use linear paths as a GMPR on the geoset. (In such cases, the constant homotopy gives an NGMPR. We will not usually mention this.)

If $n$ is even, let $v$ be a unit vector field on $S^{n-1}$. An NGMPR on $E$ is given by
\begin{equation}\Phi(a,b)(s)(t)=(1-t)a+tb+s\cdot g(t)\cdot v\bigl(\frac{b-a}{\|b-a\|}\bigr).\label{veq}\end{equation}
If $n$ is odd, let $v$ be a unit vector field on $S^{n-1}-\{e_1\}$ and let $$E'=\{(a,b):b-a=ke_1,k>0\}.$$ Then (\ref{veq}) is an NGMPR on $E-E'$, while $\Phi(a,b)(s)(t)=(1-t)a+tb+s\cdot g(t)\cdot e_2$ is an NGMPR on $E'$. Thus  $\GC(\R^n-\{0\})\le 1$ (resp.~2) if $n$ is even (resp.~odd). Equality follows from the well-known value (\cite{Far}) of $\TC(S^{n-1})$ since $\R^n-\{0\}\simeq S^{n-1}$.\end{proof}

Recall that two subsets are {\em topologically disjoint} if the closure of each is disjoint from the other. Then continuous maps on each  can be combined on the union. This notion will be used in the next and subsequent proofs.

\begin{thm}\label{Q} For $n\ge2$, let $X=\R^n-Q$, where $Q$ is a finite set with at least two points. Then $\GC(X)=\TC(X)=2$.\end{thm}
\begin{proof} Again $\Xb=\R^n$.
First let $n$ be even.
We split the nogeo set as $E_1\cup E_2$, where $E_1$ is the set of $(a,b)$ for which exactly one point of $Q$ lies on the segment $ab$, while $E_2$ is those for which two or more points of $Q$ lie on the segment. An NGMPR on $E_1$ is given by
\begin{equation}\label{abe}\Phi(a,b)(s)(t)=(1-t)a+tb+\delta\cdot s\cdot g(t)\cdot v\bigl(\frac{b-a}{\|b-a\|}\bigr),\end{equation}
where $\delta=\delta(a,b)=\frac12\min(1,d(ab,Q-(Q\cap ab)))$, a continuous function on $X\times X$.  If it is not the case that $Q$ is collinear, then the function (\ref{abe}) works on $E_2$, as well.
Note that this function would not be continuous on $E_1\cup E_2$ because of the role of $\delta$. [\![If $(a_n,b_n)\in E_1$,  with all segments $a_nb_n$ passing through the same point $x\in Q$, have the property that $a_n\to a$, $b_n\to b$, and $d(a_nb_n,x')\to0$ for some $x'\in Q\cap ab$ with $x'\ne x$, then $\delta(a_n,b_n)\to 0$ but $(a,b)\in E_2$ has $\delta(a,b)\ne0$.]\!] If $Q$ is collinear, and $q_1$ and $q_2$ are two of the points in $Q$, then we can use
$$\Phi(a,b)(s)(t)=(1-t)a+tb+s\cdot g(t)\cdot v\bigl(\frac{q_2-q_1}{\|q_2-q_1\|}\bigr)$$
on $E_2$.

If $n$ is odd, let $V$ be a unit vector which is not realizable as $(x'-x)/\|x'-x\|$ for any $x,x'\in Q$, and let $v$ be a unit vector field on $S^{n-1}-\{V\}$. Let $E_1$ be the set of $(a,b)$ for which exactly one point of $Q$ lies on the segment $ab$, and $(b-a)/\|b-a\|\ne V$, and define $\Phi$ on $E_1$ using (\ref{abe}). Let $E_2$ be the set of $(a,b)$ for which two or more points of $Q$ lie on $ab$ or exactly one point of $Q$ lies on $ab$ and $(b-a)/\|b-a\|= V$. These two portions of $E_2$ are topologically disjoint. Define $\Phi(a,b)$ on $E_2$ using (\ref{abe}), interpreting $v(V)$ to be any particular vector orthogonal to $V$. Modify as above if $Q$ is collinear.

Using linear geodesics on the geoset, we obtain $\GC(X)\le 2$, and it must equal 2 since $\TC(X)=2$ (\cite[Lemma 10.2]{Far2}). \end{proof}

\begin{thm} For $n\ge2$, $\GC(F(\R^n,2))=\TC(F(\R^n,2))=\begin{cases}1&n\text{ even}\\ 2&n\text{ odd.}\end{cases}$\end{thm}
\begin{proof} We have $\Xb=\R^n\times\R^n$. Since $F(\R^n,2)\simeq S^{n-1}$, it suffices to prove the upper bound.
As noted in \cite{D}, the nogeo set $E$ is $\{((a,a'),(b,b')):b-b'=\l(a-a'),\ \l<0\}$, and we use linear geodesics on its complement.

 If $n$ is even, an NGMPR on $E$ is given by
\begin{equation}\label{bb}\Phi((a,a'),(b,b'))(s)(t)=\bigl((1-t)a+tb,(1-t)a'+tb'+s\cdot g(t)\cdot v\bigl(\frac{b'-b}{\|b'-b\|}\bigr)\bigr).\end{equation}
\ni For a point on a path in the homotopy of (\ref{bb}) to have both components equal would require that $v\bigl(\frac{b'-b}{\|b'-b\|}\bigr)$ is a scalar multiple of $b'-b$, which cannot happen.

If $n$ is odd, decompose the nogeo set into subsets determined by whether or not  $(b'-b)/\|b'-b\|=e_1$.
We use a vector field on $S^{n-1}-\{e_1\}$ in (\ref{bb}) for one, and can replace the $v(-)$ expression by $e_2$ for the other.
\end{proof}
\begin{thm}\label{Rx} If $X=F(\R^n-\{x_0\},2)$, then $\GC(X)=\TC(X)=\begin{cases}3&n\text{ even}\\4&n\text{ odd.}\end{cases}$\end{thm}
\begin{proof} Again $\Xb=\R^n\times\R^n$.
 We say that paths $\g$ and $\g'$ {\it collide} if $\g(t)=\g'(t)$ for some $t\in I$. As noted in \cite{D}, segments $ab$ and $a'b'$ collide iff $b'-b=\l(a'-a)$ for some $\l<0$.

 Let $n$ be even. We partition $X\times X$ into  six sets, on each of which we will define a GMPR or NGMPR. Sets $C_0$, $C_1$, and $C_{x_0}$ consist of those $((a,a'),(b,b'))$ for which segments $ab$ and $a'b'$ collide and ($C_0$) neither segment contains $x_0$, ($C_1$) one of the segments contains $x_0$ and the other segment has positive length, and ($C_{x_0}$) they collide at $x_0$. Sets $E_j$, $j\in\{0,1,2\}$, consist of those $((a,a'),(b,b'))$ which do not collide and $j$ of the segments $ab$ and $a'b'$ contain $x_0$. The set $L_1$ is those $((a,a'),(b,b'))$ such that $x_0$ lies on $a'b'$ as does $a=b$, or $x_0$ lies on $ab$ as does $a'=b'$.

 Note that $C_0$ and $E_1$ are topologically disjoint, as are $C_1$ and $E_2$, and also $C_{x_0}$ and $L_1$. Indeed, each set has a property, preserved under limits, which is not true of any element of the paired set. Once we have noted the GMPR and NGMPRs on each of the seven regions when $n$ is even, the domains $E_0$, $C_0\cup E_1$, $C_1\cup E_2$, and $C_{x_0}\cup L_1$ imply $\GC(X)\le 3$ when $n$ is even. Since $\TC(X)=3$ when $n$ is even (\cite{GG}), we obtain the result in this case.

 We use the linear geodesic on $E_0$. For $E_1$ and $E_2$, let $\delta=\min(1,d(ab,a'b'))$, where $d(ab,a'b')$ is the minimum distance for corresponding values of $t$.
 On $E_2$, we use
 \begin{eqnarray}\nonumber\Phi((a,a'),(b,b'))(s)(t)&=&\biggl((1-t)a+tb+\tfrac13\delta \cdot s\cdot g(t)\cdot v\bigl(\frac{b-a}{\|b-
a\|}\bigr),\\
&&(1-t)a'+tb'+\tfrac13\delta \cdot s\cdot g(t)\cdot v\bigl(\frac{b'-a'}{\|b'-
a'\|}\bigr)\biggr).\label{two}\end{eqnarray}
Because both $ab$ and $a'b'$ pass through $x_0$ and do not collide, both $b-a$ and $b'-a'$ must be nonzero. On $E_1$, we modify this formula by removing the $\frac13\delta s gv$ term in the component which did not pass through $x_0$.

On $C_0$, we use the NGMPR of (\ref{bb}) with the $s\cdot g(t)$ multiplied by an additional factor $\frac12\min(1,d(ab,x_0))$.
On $C_1$, we use
\begin{equation}\label{Y1}\Phi((a,a'),(b,b'))(s)(t)=\biggl((1-t)a+tb,(1-t)a'+tb'+ s\cdot g(t)\cdot \tfrac{b-a}{\|b-a\|}\biggr),\end{equation}
when $a'b'$ passes through $x_0$, and a similar formula when $ab$ passes through $x_0$.  These curves do not pass through $x_0$ since $x_0$ cannot be written as the sum of a point on $a'b'$ plus a nonzero multiple of $b-a$. Since $b-b'$ is a scalar multiple of $a-a'$, if the components of (\ref{Y1}) were to collide, $b-a$ would be a nonzero multiple of $a-a'$, which it is not.

On $L_1$, we use (\ref{bb}) when $x_0$ lies on $a'b'$, as does $a=b$, with an obvious modification if $ab$ and $a'b'$ play opposite roles.
On $C_{x_0}$ we use, similarly to $C_1$,
\begin{eqnarray}\Phi((a,a'),(b,b'))(s)(t)&=&\biggl((1-t)a+tb+s\cdot g(t)\cdot \tfrac{b'-b}{\|b'-b\|},\nonumber\\
&&(1-t)a'+tb'+ s\cdot g(t)\cdot\tfrac{a-b}{\|a-b\|}\biggr).\label{Y2}\end{eqnarray}
For these to collide, we would need $ab$ to be parallel to $aa'$, which it isn't. Also note that $a-b$ is nonzero since otherwise we would have $a=b=x_0$, which cannot happen.

When $n$ is odd, we no longer have a vector field on $S^{n-1}$. For the cases that used such a vector field, we use a vector field on $S^{n-1}-\{\text{pt}\}$. By choosing the excluded point differently in different cases, we can arrange it so that all the excluded cases are topologically disjoint, and so can be combined into one additional domain, again agreeing with the known result for TC. Indeed, choose  vectors $V_i$, $1\le i\le 4$, in $S^{n-1}$ such that if $i\ne j$, then $V_i\ne \pm V_j$. Our fifth domain is
\begin{eqnarray*}&&(C_0\cap \{b'-b\in\la V_1\ra\})\cup(E_1\cap\{b-a\text{ or }b'-a'\in\la V_2\ra\})\\
&&\cup(E_2\cap\{b-a\text{ or }b'-a'\in\la V_3\ra\})\cup (L_1\cap \{b'-b\in\la V_4\ra\}),\end{eqnarray*}
where $\la V\ra$ denotes the span of a vector $V$.
\end{proof}
\begin{thm} For the unordered configuration space $X=C(\R^2-x_0,2)$,
$$\GC(X)=\TC(X)=2.$$ \end{thm}
\begin{proof} We begin by looking at $C(\R^n-x_0,2)$, and then will specialize to $n=2$. The metric on $C(\R^n,2)$ is, as in \cite{D},
$$d(\{a,a'\},\{b,b'\})=\min(d((a,a'),(b,b')),d((a,a'),(b',b))),$$
and we use this on the subspace $C(\R^n-x_0,2)$.
Here we use results from \cite[Prop 4.3,(4.1)]{D} that
 in $\R^n\times \R^n$,  $d((a,a'),(b,b'))=d((a,a')(b',b))$  iff $aa'\perp bb'$, and if $d((a,a'),(b,b'))<d((a,a')(b',b))$, then $ab$ and $a'b'$ do not collide. Also, $\Xb=\R^n\times\R^n/(a,a')\sim(a',a)$.

 We first consider pairs $(\{a,a'\},\{b,b'\})$ with $d((a,a'),(b,b'))\ne d((a,a')(b',b))$, and label them so that $d((a,a'),(b,b'))<d((a,a')(b',b))$. Let $E_0$ denote the set of those for which neither $ab$ nor $a'b'$ passes through $x_0$ and use the linear GMPR on $E_0$. Let $E_1$ denote the set of those for which exactly one of $ab$ and $a'b'$ passes through $x_0$, and, if $x_0\in ab$, assuming now that $n$ is even so that there is a unit vector field $v$ on $S^{n-1}$, use
 \begin{equation}\label{another}\Phi(\{a,a'\},\{b,b'\})(s)(t)=\bigl((1-t)a+tb+\delta\cdot s\cdot g(t)\cdot v\bigl(\tfrac{a-b}{\|a-b\|}\bigr),(1-t)a'+tb'\bigr)\end{equation}
  with $\delta=\frac12\min(1,d(ab,a'b'))$, with obvious reversal if, instead, $x_0\in a'b'$. Let $E_2$ denote those for which $ab$ and $a'b'$ both pass through $x_0$ (but not for the same $t$), and use (\ref{two}) as the NGMPR.

Now we consider pairs $(\{a,a'\},\{b,b'\})$ for which $aa'\perp bb'$, so  $d((a,a'),(b,b'))=d((a,a')(b',b))$. Note that $ab$ and $a'b'$ do not collide, nor do $ab'$ and $a'b$.
We now specialize to $n=2$. Choose $ab$ so that $\vec{bb'}$ is a 90-degree counterclockwise rotation from $\vec{aa'}$.
 Let $Y_0$ denote the set of those for which none of the segments $ab$, $a'b$, $ab'$, and $a'b'$ pass through $x_0$.  A GMPR on $Y_0$ can be obtained by using the linear path from $(a,a')$ to $(b,b')$. Let $Y_1$ denote the set of those such that exactly one of the pairs $\{ab,a'b'\}$ and $\{ab',a'b\}$ has neither segment passing through $x_0$. Use the linear path on that pair as a GMPR. For example, if $x_0$ lies on $ab'$, use the linear path from $(a,a')$ to $(b,b')$. Let $Y_2$ denote the set of those for which  the pairs $\{ab,a'b'\}$ and $\{ab',a'b\}$ have one segment each passing through $x_0$. The situation will be like that in Figure \ref{figz}, or with $a$ and $b$ reversed, or with primed and unprimed reversed. These different situations are topologically disjoint.

\bigskip
\begin{minipage}{6in}
\begin{fig}\label{figz}

{\bf A typical $Y_2$ configuration}

\begin{center}

\begin{\tz}[scale=.65]
\draw (0,2) -- (0,0) -- (6,0);
\node at (0,-.5) {$a$};
\node at (2,-.5) {$x_0$};
\node at (4,-.5) {$b$};
\node at (6.4,0) {$b'$};
\node at (0,2.4) {$a'$};
\node at (0,0) {$\ssize\bullet$};
\node at (2,0) {$\ssize\bullet$};
\node at (4,0) {$\ssize\bullet$};
\node at (6,0) {$\ssize\bullet$};
\node at (0,2) {$\ssize\bullet$};
\end{\tz}
\end{center}
\end{fig}
\end{minipage}

\bigskip

 We can use (\ref{another}) as an NGMPR on diagrams in $Y_2$ of the form in Figure \ref{figz}, with obvious modifications for its variants.

We can use $E_0$, $E_1\cup Y_0$, $E_2\cup Y_2$, and $Y_1$ as our four domains, since $E_1$ and $Y_0$ are topologically disjoint, as are $E_2$ and $Y_2$.

\end{proof}

\section{$F(\R^n-Q,2)$ with $2\le|Q|<\infty$}\label{sec4}
\begin{thm}\label{Qthm} Let $n$ be even, and let $Q$ be a finite subset of $\R^n$ with at least two points, and $X=F(\R^n-Q,2)$. Then $\GC(X)\le 5$.\end{thm}
By \cite{GG}, $\TC(F(\R^n-Q,2))=4$. We think that is likely that, at least for $n=2$, our result for $\GC(X)$ is sharp, which would give a nice example of $\TC<\GC$.
\begin{proof}
 We partition $X$ into 18 subsets, on each of which there is a GMPR or NGMPR. Then we will group them into six collections of topologically disjoint subsets.

We use the word `collide'' as in the proof of Theorem \ref{Rx}. Recall that $ab$ and $a'b'$ collide iff $b-b'$ is a negative multiple of $a-a'$.
There are sets $E_0$, $E_1$, $E_2$, $E_{1,1}$, $E_{1,2}$, and $E_{2,2}$, in which the segments $ab$ and $a'b'$ do not collide, and the subscripts indicate how many points of $Q$ lie on each segment, with ``2'' referring to ``2 or more.'' For example, $E_2$ consists of those $((a,a'),(b,b'))$ for which $ab$ and $a'b'$ do not collide and one of these segments contains two or more points of $Q$, while the other has none. If the segments intersect at a point of $Q$ (for differing values of the parameter $t$), then that point counts for both lines. So, for example, $E_{1,1}$ consists both of noncolliding elements where the segments do not meet at a point of $Q$ and each contains a point of $Q$, and those where the two segments meet at a point of $Q$, and neither contains other points of $Q$. If $Q$ has only two or three points, some of these sets can be empty.

There are also sets $C_0$, $C_1$, $C_2$, $C_{1,1}$, $C_{1,2}$, and $C_{2,2}$, in which the segments collide, but not at a point of $Q$, and the subscripts have the same meaning as before. For these, there is not the issue of classifying what happens if segments intersect at a point of $Q$. For these $C$-sets, we exclude colliding elements in which $a$, $a'$, $b$, and $b'$ are collinear.

Next we have sets $Y_j$, $j=0,1,2$, in which the segments collide at a point of $Q$, and $j$ of the segments contain one or more additional points of $Q$. Again we exclude the case in which $a$, $a'$, $b$, and $b'$ are collinear. Finally we have the linear cases $L_j$, $j=0,1,2$, in which the four points are collinear and $aa'$ and $bb'$ have opposite directions, and $j$ points of $Q$ lie on $ab\cup a'b'$.

On $E_0$, we use the linear geodesic.
On the other $E$ sets, we use formulas like (\ref{two}), using just the linear part when a segment does not contain any points of $Q$, and modifying $\delta$ to equal $\min(1,d_1,d_2)$, where $d_1$ is the distance between the parametrized segments, and $d_2$ is the distance from the segment to the nearest point of $Q$ not on it.

Formulas for near geodesics on the $C$ and $Y$ sets are similar to those that worked for the $C$ sets in the proof of Theorem \ref{Rx}. We use a factor $\delta_1=\frac12\min(1,d(ab,Q-(Q\cap ab)))$ on the first component, and an analogue on the second. Incorporating that, we use an analogue of (\ref{bb}) on $C_0$, of (\ref{Y1}) on $C_1$ and $C_2$, and of (\ref{Y2}) on $C_{1,1}$, $C_{1,2}$,  $C_{2,2}$, and each $Y_j$. On each set $L_j$, we can use
\begin{eqnarray*}\Phi((a,a'),(b,b'))(s)(t)&=&\biggl((1-t)a+tb+\delta_2 \cdot s\cdot g(t)\cdot v\bigl(\frac{b'-b}{\|b'-
b\|}\bigr),\\
&&(1-t)a'+tb'-\delta_2 \cdot s\cdot g(t)\cdot v\bigl(\frac{b'-b}{\|b'-
b\|}\bigr)\biggr).\end{eqnarray*}

We can group these into six collections of topologically disjoint subsets as follows.
\begin{eqnarray*}&&E_0,\quad E_1\cup C_0,\quad E_2\cup E_{1,1}\cup  C_1,\\
&&E_{1,2}\cup C_{2}\cup C_{1,1}\cup Y_0\cup L_0,\\
&&E_{2,2}\cup C_{1,2}\cup Y_1\cup L_1,\quad C_{2,2}\cup Y_2\cup L_2.\end{eqnarray*}
To show two sets are topologically disjoint, we usually show that each has a property which is preserved under limits of sequences and is not shared by any element of the other set. For example, in the fifth of the above sets, we could use ``collinear'' for $L_1$ and ``collide at a point of $Q$'' for $Y_1$.  The limit of a sequence of $E_{2,2}$ sets could collide, but would be either in $C_{2,2}$ or $Y_2$. Similarly, the limit of a sequence of $C_{1,2}$ sets could be in $Y_2$, but not in $Y_1$. Since two points determine a line, the limit of a sequence of $C_{2,2}$ sets cannot be in $Y_2$.

\end{proof}

\section{Configuration spaces of graphs}
Configuration spaces $F(G,2)$ of graphs $G$, as studied in \cite{DHR}, are handled quite differently than the cases considered above. In \cite{DHR}, $F(G,2)$ was given the subspace metric from $G\times G$, where $G\times G$ often had the Euclidean metric, using distance in the graph $G$. Thus $d((a,a'),(b,b'))=\sqrt{d(a,b)^2+d(a',b')^2}$.

When $X=F(G,2)$ in the above metric, the completion $\Xb$ equals $G\times G$ and is geodesic. However, certain geodesics in $\Xb$ cannot be approximated by paths in $X$. For example, let $G$ be the $Y$-graph $Y$ with essential vertex $v$, and suppose $a$, $a'$, $b'$, and $b$ are on the same edge at distance $1$, $1+\delta$, $1+2\delta$, and $1+3\delta$, respectively, from the vertex $v$.

The geodesic in $\Xb$ from $(a,a')$ to $(b,b')$ is the linear path of length $\delta\sqrt{10}$. However, since direct motion will involve a collision, there is no path in $X$ whose length is  close to this. A short path in $X$ from $(a,a')$ to $(b,b')$ is one that moves from $(a,a')$ back just beyond $v$ onto the two empty arms, and from there to $(b,b')$, with length slightly greater than $\sqrt{1+(1+3\delta)^2}+\sqrt{(1+\delta)^2+(1+2\delta)^2}$. Thus there is no near-geodesic from $(a,a')$ to $(b,b')$.

Since the definition of GC in \cite{RM} only applies well to geodesic spaces, to consider $\GC(F(Y,2))$ in \cite{DHR}, we replaced $F(Y,2)$ by the homotopically equivalent subspace $F_\eps(Y,2)$ which consisted of points $(a,a')$ satisfying $d(a,a')\ge\eps$ for some fixed positive number $\eps$. For example, in the diagram at the left in Figure \ref{fig5}, there is no geodesic in $F(Y,2)$ from $(a,a')$ to $(b,b')$ because in the linear motion from $(a,a')$ to $(b,b')$, the first particle would overtake the second, which is not allowed. In \cite{DHR}, we represented paths in the graph by paths in the $xy$ plane, where, in this case, the $x$-axis corresponds to motion of the first particle on the two upper arms on the graph, with the vertex at 0, and the $y$-axis, similarly, corresponds to motion of the second particle on the arms on the left. This is shown on the right side of Figure \ref{fig5}, in which the interior of the shaded region is excluded, as those points do not satisfy $d(a,a')\ge\eps$. The representation of the geodesic in $F_\eps(Y,2)$ is indicated. It corresponds to the path in the graph which goes from $(a,a')$ to the point $(v,\eps)$, with the $\eps$ on the bottom arm, and from there to $(b,b')$. Here $\eps<\min(d(a,a'),d(v,b'))$.

\bigskip
\begin{minipage}{6in}
\begin{fig}\label{fig5}

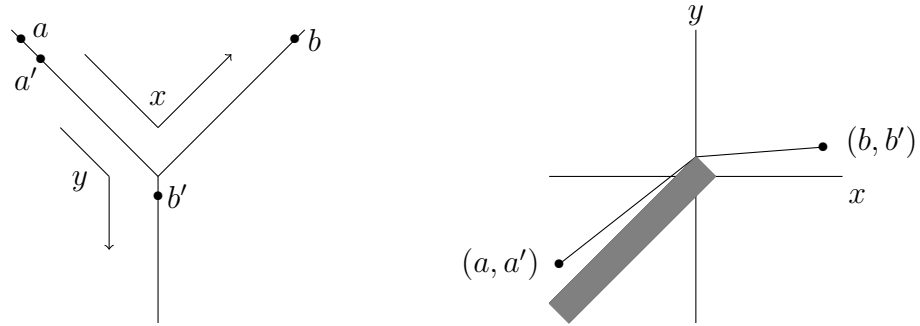
{\bf An element in $F_\eps(Y,2)$ and a representation of its path}

\begin{center}

\begin{\tz}[scale=.65]
\draw (0,0) -- (0,3) -- (3,6);
\draw (0,3) -- (-3,6);
\node at (-2.8,5.8) {$\ssize\bullet$};
\node at (-2.4,5.4) {$\ssize\bullet$};
\node at (0,2.6) {$\ssize\bullet$};
\node at (2.8,5.8) {$\ssize\bullet$};
\node at (-2.4,6) {$a$};
\node at (-2.7,5) {$a'$};
\node at (.4,2.6) {$b'$};
\node at (3.2,5.8) {$b$};
\draw (8,3) -- (14,3);
\draw (11,0) -- (11,6);
\node at (11,6.3) {$y$};
\node at (14.3,2.6) {$x$};
\draw (8,.4) -- (11,3.4) -- (11.4,3) -- (8.4,0);
\draw (8.2,1.2) -- (11,3.4) -- (13.6,3.6);
\node at (7,1.2) {$(a,a')$};
\node at ((14.8,3.7) {$(b,b')$};
\node at (8.2,1.2)  {$\ssize\bullet$};
\node at (13.6,3.6)  {$\ssize\bullet$};
\filldraw [color=gray] (8,.4) -- (11,3.4) -- (11.4,3) -- (8.4,0)-- (8,.4);
\draw (-2,4) -- (-1,3);
\draw [->] (-1,3) -- (-1,1.5);
\node at (-1.6,2.9) {$y$};
\draw (-1.5,5.5) -- (0,4);
\draw [->] (0,4) -- (1.5,5.5);
\node at (0,4.6) {$x$};
\end{\tz}
\end{center}
\end{fig}
\end{minipage}

\bigskip
For our new approach, we use the {\em intrinsic metric} $d_I$ on $X=F(G,2)$, defined by $d_I(x,y)$ is the infimum of the $d$-lengths of paths in $X$ from $x$ to $y$. For the spaces considered here, this metric induces the same topology as does the original metric since if $\eps<d(a,a')/\sqrt2$, the $\eps$-balls around $(a,a')$ in the two topologies are equal. This is true because linear motion between points $(b,b')$ in the $\eps$-ball and $(a,a')$ avoids collision. This also implies that lengths of paths using the $d_I$ metric equal the $d$-length because length is determined from arbitrarily small segments.  Since the intrinsic and Euclidean topologies are the same, we again have $\TC(F(G,2))\le \GC(F(G,2))$, where $\GC(F(G,2))$ is defined as in Definition \ref{NGMPRdef}, using the metric $d_I$.

Let $F_I(Y,2)$ denote $F(Y,2)$ in the intrinsic metric. To form the completion, we adjoin points $(x,x)_1$ and $(x,x)_2$ for all $x\in Y-\{v\}$, and $(v,v)$. Here $(x,x)_1$ (resp.~$(x,x)_2$) is the limit of Cauchy sequences of points $(a_n,a_n')$ with $a_n,a_n'\to x$ in $Y$ and $d(a_n,v)>d(a_n',v)$ (resp.~$d(a_n,v)<d(a_n',v)$). The reason for having both is that, as noted above, if these points are all very close to $x$ and $d(a,v)<d(a',v)$ and $d(b',v)<d(b,v)$, then $d_I((a,a'),(b,b'))$ is not small. Note that $d_I((x,x)_1,(x,x)_2)=2\sqrt2d(x,v)$ since it equals $\ds\lim_{n\to\infty}d_I((a_n,a_n'),(b_n,b_n'))$ when all four points approach $x$, with $d(a_n,v)>d(a_n',v)$ and $d(b_n,v)<d(b_n',v)$.
This issue is not present for sequences approaching $(v,v)$. If $a,a',b',b$ are on the same arm at distance $\delta,2\delta,3\delta,4\delta$, respectively, from $v$, then $d_I((a,a'),(b,b'))=\delta(\sqrt{1^2+2^2}+\sqrt{3^2+4^2})$, so a Cauchy sequence of such points can approach $(v,v)$.

If you tried to follow the linear path from $(a,a')$ to $(b,b')$ in the situation on the left side of Figure \ref{fig5}, $(a,a')$ will quickly get to a point $(x,x)_1$. If the linear motion is attempted beyond that, then $a$ is now closer to $v$ than is $a'$, so the path is not continuous in the $d_I$-metric.
The geodesic in $\overline{F_I(Y,2)}$ is  the path from  $(a,a')$ to $(v,v)$ to $(b,b')$.  Our near geodesic is a homotopy which at parameter $s=\eps$ could be chosen to be represented by the path on the right side of Figure \ref{fig5}.

Let $F_I=F_I(Y,2)$. Our analysis of $\GC(F_I)$ is patterned after the analysis of $\GC(F_\eps(Y,2))$ in \cite{DHR}. The three arms of the $Y$-graph are referred to as  closed or open depending on whether or not the vertex $v$ is included. We will assume that the arms have length $\ge2$. If the four points $a$, $a'$, $b$, and $b'$ are all on one or two arms, then their orientation (same or opposite) refers to $\vec{aa'}$ compared with $\vec{bb'}$. We partition $F_I\times F_I$ into the following six sets.
\begin{itemize}
\item[$C_{1,1,2}$:] Three open arms occupied, and direct motion involves a collision. (The fourth point cannot be at $v$.)
\item[$X_{1,1,2}$:] Three open arms occupied, direct motion does not involve a collision. (The fourth point possibly at $v$.)
\item[$C_4$:] All four points are on a single closed arm, with opposite orientation.
\item[$C_{2,2}$:] Two points on each of two open arms, with opposite orientation.
\item[$C_{3,1}$:] Three points on one closed arm, the other on a different (open) arm, with opposite orientation.
\item[$L$:] All are on one or two arms, with the same orientation.
\end{itemize}

Our two domains are
$$E_0=C_{1,1,2}\cup C_4\cup C_{2,2}, \qquad E_1=X_{1,1,2}\cup C_{3,1}\cup L.$$
We number the three arms cyclically in clockwise order.
We first describe the NGMPR on $E_0$.

Here we introduce the notation that $s_i$ denotes the point at distance $s$ from $v$ on arm $i$, and $0$ denotes $v$. Arrows denote uniform linear motion along the graph, and the overall path is parametrized by arc length in $Y\times Y$ (in the $d$-metric).

The NGMPR on $C_4$ is: if they are all on arm $i$ with $d(a,v)<d(a',v)$, then
$$(a,a')\to (s_{i+1},0)\to (0,s_{i+2})\to (b,b'),$$
while if $d(a',v)<d(a,v)$, then
$$(a,a')\to (0,s_{i+1})\to (s_{i+2},0)\to (b,b').$$

On $C_{1,1,2}$, the NGMPR is as follows. If $a$ is alone on arm $i$ and $a'$ is alone on a different arm, and $d(b',v)>d(b,v)$, or if $b$ is alone on arm $i$ and $b'$ is alone on a different arm, and $d(a',v)>d(a,v)$, we use
$$(a,a')\to (s_i,0) \to (b,b').$$
Such configurations cannot approach an element of $C_4$.
Otherwise, if $a'$ is alone on an open arm, so $a$ and $b'$ are on another arm $i$, we use
$$(a,a')\to (0,s_{i+1})\to (s_{i+2},0)\to (b,b'),$$
while if $a$ is alone on an open arm, so $a'$ and $b$ are on another arm $i$, we use
$$(a,a')\to (s_{i+1},0)\to (0,s_{i+2})\to (b,b').$$
This NGMPR agrees with that of the limit in $C_4$ of such configurations as $a'$ and $b$ approach $v$.
 If $a$ and $a'$ are on the same open arm $i$ with $d(a,v)>d(a',v)$, we can use $(a,a')\to (s_i,0)\to (b,b')$.

On $C_{2,2}$, we consider first the case in which $a$ and $a'$ lie on one arm and $b$ and $b'$ on another, with $a'$ and $b'$ closer to $v$. A sequence of such diagrams cannot converge to an element of $C_{1,1,2}\cup C_4$. If $i$ is the free arm, we can use $(a,a')\to (0,s_i)\to (b,b')$. The other type of $C_{2,2}$  diagram is a bit more complicated. We just consider the case when $a$ and $b'$ are on open arm 1 and $a'$ and $b$ are on open arm 2. Our rule must satisfy that the limit as $a$ and $b'$ approach $v$ should be
$$(a,a')\to (s_3,0)\to (0,s_1)\to (b,b'),$$
while the limit as $a'$ and $b$ approach $v$ should be
$$(a,a')\to (0,s_2)\to (s_3,0)\to (b,b').$$
We consider the portion from $(a,a')$ to $(s_3,0)$. The portion from $(s_3,0)$ to $(b,b')$ can be treated in an analogous way, looking backward from $(b,b')$.

Let $d=d(a,v)$ and $d'=d(a',v)$. We use $(a,a')\to (s_3,0)$ if $d'>s$ or $d'>d$, and we use $(a,a')\to (0,s_2)\to (s_3,0)$ if $d'<s/2$ and $d'<d/2$. In Figure \ref{figa}, the $x$ (resp.~$y$)-axis represents the motion of the first (resp.~second) particle for a fixed value of the homotopy parameter $s$, somewhat similarly to Figure \ref{fig5}. The starting position $(d,d')$ is a point in the first quadrant, and we wish to represent the path to the point $(-s,0)$, corresponding to the point $(s_3,v)$ on the $Y$-graph on the left.
For starting positions above the intermediate region $R$, we move along the straight segment from $(d,d')$ to $(-s,0)$. For starting positions below $R$, we move from $(d,d')$ to $(0,s)$ and then to $(-s,0)$. This satisfies the requirement regarding the behavior as $a$ or $a'$ approaches $v$. Note that we need not be concerned with the limit as both $a$ and $a'$ approach 0, as the point $(v,v)$ is not in $F(Y,2)$.

For a point in the intermediate region $R$, we move from $(d,d')$ to $(0,y)$ and then to $(-s,0)$, where $y$ is defined as follows.
Let $y_0$ (resp.~$y_1$) be the smallest (resp.~largest) $y$ value of the points in $R$ with $x=d$. Let $(0,y_2)$ be the intersection of the segment from $(d,y_1)$ to $(-s,0)$ with the $y$-axis. If $d'=(1-t)y_0+ty_1$, then $y=(1-t)s+ty_2$. This choice of paths varies continuously with nonnegative values of $d$, $d'$, and $s$, provided $d+d'>0$.

\bigskip
\begin{minipage}{6in}
\begin{fig}\label{figa}

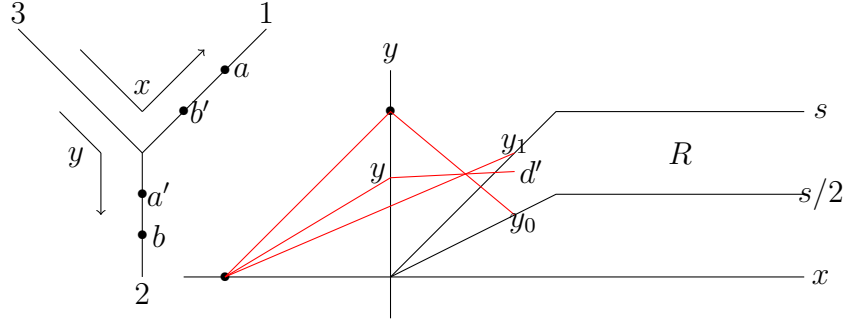
{\bf An element in $C_{2,2}$ and a representation of its paths}

\begin{center}

\begin{\tz}[scale=.55]
\draw (0,0) -- (0,3) -- (3,6);
\draw (0,3) -- (-3,6);
\draw (-2,4) -- (-1,3);
\draw [->] (-1,3) -- (-1,1.5);
\node at (-1.6,2.9) {$y$};
\draw (-1.5,5.5) -- (0,4);
\draw [->] (0,4) -- (1.5,5.5);
\node at (0,4.6) {$x$};
\draw (1,0) -- (16,0);
\node at (16.4,0) {$x$};
\node at (6,5.4) {$y$};
\draw (6,-1) -- (6,5);
\draw (6,0) -- (10,4) -- (16,4);
\draw (6,0) -- (10,2) -- (16,2);
\node at (13,3) {$R$};
\node at (16.4,4) {$s$};
\node at (16.4,2) {$s/2$};
\node at (2,0) {$\ssize\bullet$};
\node at (6,4) {$\ssize\bullet$};
\draw [color=red] (9,1.5) -- (6,4) -- (2,0);
\draw [color=red] (9,3) -- (2,0);
\draw [color=red] (9,2.55) -- (6,2.4) -- (2,0);
\node at (5.7,2.5) {$y$};
\node at (9.2,1.3) {$y_0$};
\node at (9,3.2) {$y_1$};
\node at (9.4,2.55) {$d'$};
\node at (3,6.4) {$1$};
\node at (0,-.4) {$2$};
\node at (-3,6.4) {$3$};
\node at (1,4) {$\ssize\bullet$};
\node at (2,5) {$\ssize\bullet$};
\node at (0,1) {$\ssize\bullet$};
\node at (0,2) {$\ssize\bullet$};
\node at (2.4,5) {$a$};
\node at (1.4,4) {$b'$};
\node at (.4,1) {$b$};
\node at (.4,2) {$a'$};
\end{\tz}
\end{center}
\end{fig}
\end{minipage}

\bigskip

This completes the description of the NGMPR on $E_0$.
We next describe the NGMPR on $E_1$. On $L$, we use linear motion $(a,a')\to (b,b')$ and the constant homotopy.

The NGMPR on $C_{3,1}$ is, at homotopy parameter $s$, given as follows. If $a$, $a'$, and $b$ are on arm $i$ with $d(a,v)<d(a',v)$, and $b'$ is on arm $j$, and arm $k$ is free,
or if $b$, $b'$, and $a$ are on arm $j$ with $d(b,v)<d(b',v)$, and $a'$ is on arm $i$, and arm $k$ is free, then
$$(a,a')\to (0,s_i)\to (s_k,0) \to (0,s_j)\to (b,b').$$
Reversing primed and unprimed reverses the coordinates of the  parts involving 0 and $s$.
This has the property that if $a$ and $b$ are at the vertex, with $a'$ and $b'$ on distinct open arms, the same motion is obtained if it is thought of as $a$, $a'$, and $b$ on an arm or as $b$, $b'$, and $a$ on an arm.

On $X_{1,1,2}$, we would like to use linear motion and the constant homotopy. However, for noncolliding diagrams of any of the three types on the left side of Figure \ref{figc}, this homotopy would not agree, in the limit, with that in the $C_{3,1}$ diagram on the right side, as $a'$ and $b'$ approach $v$. In the first two diagrams, the two letters on an arm may appear in either order.

\bigskip
\begin{minipage}{5.5in}
\begin{fig}\label{figc}

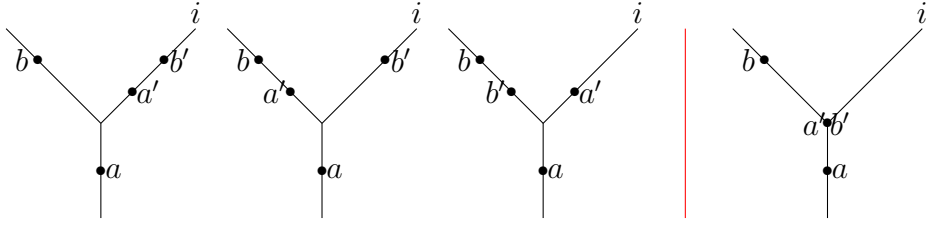
{\bf Three elements of $X_{1,1,2}$ and a limiting element}

\begin{center}

\begin{\tz}[scale=.42]
\draw (0,0) -- (0,3) -- (3,6);
\draw (0,3) -- (-3,6);
\draw (7,0) -- (7,3) -- (10,6);
\draw (7,3) -- (4,6);
\draw (14,0) -- (14,3) -- (17,6);
\draw (14,3) -- (11,6);
\draw (23,0) -- (23,3) -- (26,6);
\draw (23,3) -- (20,6);
\draw [color=red] (18.5,0) -- (18.5,6);
\node at (-2,5) {$\ssize\bullet$};
\node at (-2.5,5) {$b$};
\node at (4.5,5) {$b$};
\node at (11.5,5) {$b$};
\node at (20.5,5) {$b$};
\node at (2,5) {$\ssize\bullet$};
\node at (2.5,5) {$b'$};
\node at (9.5,5) {$b'$};
\node at (1,4) {$\ssize\bullet$};
\node at (1.5,4) {$a'$};
\node at (15.5,4) {$a'$};
\node at (0,1.5) {$\ssize\bullet$};
\node at (.4,1.5) {$a$};
\node at (7.4,1.5) {$a$};
\node at (14.4,1.5) {$a$};
\node at (23.4,1.5) {$a$};
\node at (12.5,4) {$b'$};
\node at (5.5,4) {$a'$};
\node at (22.6,3) {$a'$};
\node at (23.4,3) {$b'$};
\node at (5,5) {$\ssize\bullet$};
\node at (9,5) {$\ssize\bullet$};
\node at (6,4) {$\ssize\bullet$};
\node at (7,1.5) {$\ssize\bullet$};
\node at (12,5) {$\ssize\bullet$};
\node at (15,4) {$\ssize\bullet$};
\node at (13,4) {$\ssize\bullet$};
\node at (14,1.5) {$\ssize\bullet$};
\node at (21,5) {$\ssize\bullet$};
\node at (23,3) {$\ssize\bullet$};
\node at (23,1.5) {$\ssize\bullet$};
\node at (3,6.5) {$i$};
\node at (10,6.5) {$i$};
\node at (17,6.5) {$i$};
\node at (26,6.5) {$i$};
\end{\tz}
\end{center}
\end{fig}
\end{minipage}

\bigskip

We will allow $a$, $a'$, $b$, and $b'$ to denote distance of the corresponding points from the vertex.
The NGMPR for the first of the diagrams of Figure \ref{figc} has $\Phi((a,a'),(b,b'))(s)$ equal to the path described as follows.
Let $M=\frac{\max(a',b')}{\min(a,b)}$. For $s\le M$, it is the linear path $(a,a')\to (b,b')$, which can also be written
$$(a,a')\to (0,(\tfrac b{a+b}a'+\tfrac a{a+b}b')_i) \to (b,b').$$
For $s=2M$, it is the path $(a,a')\to(0,(2M)_i)\to(b,b')$, and for $M\le s\le 2M$, it is the linear interpolation between these two paths.
Note that $s\le1$, so what this really means is that for $s\in[0,1]\cap[M,2M]$ it is the path
$(a,a')\to(0,y_i)\to (b,b')$, where
$$y=(1-\tfrac {s-M}M)(\tfrac b{a+b}a'+\tfrac a{a+b}b')+\tfrac{s-M}M\cdot2M.$$
For $4M\le s\le1$, it is the path
$$(a,a')\to (s_{i+1},0)\to (0,s_i)\to (s_{i+2},0)\to (b,b'),$$
and for $s\in[0,1]\cap[2M,4M]$, it is a linear interpolation between the paths for $2M$ and $4M$, similar to what was done in Figure \ref{figa}.

This satisfies, for all $s\in[0,1]$,
$$\lim_{a',b'\to v}\Phi((a,a'),(b,b'))(s)=\Phi((a,v),(b,v))(s),$$
where the latter is from the $\Phi$ formula on $C_{3,1}$. It also satisfies
$$\lim_{a\to v}\Phi((a,a'),(b,b'))(s)=\Phi((a,v),(b,v))(s),$$
where the latter is the constant linear homotopy from the $\Phi$ formula on $L$, and similarly as $b\to v$.

The same formulation works for the third diagram, using $M=a'/a$. This will approach the $C_{3,1}$ motion as $a'\to v$, and will approach the $L$ motion as $a\to v$. The second diagram of Figure \ref{figc} is slightly more complicated because we have to also worry about the $C_{3,1}$ motion in the limit as $a$ and $b$ approach $v$.

\def\Mbar{\overline{M}}
Let $M=\max(a,b)$ and $M'=\max(a',b')$, and $\Mbar=\min(M,M')$. For $s\le\Mbar$, we use the linear path, which can also be written
$$(a,a')\to ((\tfrac{b'}{a'+b'}a+\tfrac{a'}{a'+b'}b)_{i+1},0)\to (0,(\tfrac b{a+b}a'+\tfrac a{a+b}b')_i)\to (b,b').$$
For $s=2\Mbar$, use $(a,a')\to((2\Mbar)_{i+1},0)\to(0,(2\Mbar)_i)\to (b,b')$, and for $\Mbar\le s\le2\Mbar$, interpolate linearly between these, in the sense described above for the first diagram of Figure \ref{figc}.
For $4M\le s\le1$, the $(a,a')\to(s_{i+1},0)$ part of the path is replaced by $(a,a')\to (0,s_{i+2})\to (s_{i+1},0)$, and for $s\in[0,1]\cap[2\Mbar,4M]$, this part of the path is interpolated. Similarly on the other end, for $4M'\le s\le1$, the $(0,s_i)\to (b,b')$ part has $(s_{i+2},0)$ inserted in the middle, and we interpolate on this side for $s\in[0,1]\cap[2\Mbar, 4M']$. Only one or the other of $M$ and $M'$ can approach 0, and each will approach the appropriate $C_{3,1}$ motion in the limit.

 This completes our NGMPR on $E_1$.
 Thus we have proved the following result, which is an analogue of \cite[Theorem 1(a)]{DHR}.
\begin{thm} If $Y$ denotes the $Y$ graph, then $\GC(F_I(Y,2))=\TC(F(Y,2))=1$.\end{thm}
We see that for these graph configuration spaces, the analysis of GC is  closely related to the analysis of $\GC(F_\eps(G,2))$ in \cite{DHR}, but the perspective is quite different.

\def\line{\rule{.6in}{.6pt}}

\end{document}